\documentclass[reqno,a4paper,12pt]{amsart}
\usepackage[top=32mm, bottom=32mm, left=32mm, right=32mm]{geometry}
\usepackage{mathrsfs,enumerate,amssymb,amsmath,color,lineno}
\usepackage[all]{xy}
\usepackage{latexsym}
\usepackage{amsthm,a4wide,color}
\usepackage{amscd,verbatim}
\usepackage{graphicx}
\usepackage[colorlinks=true]{hyperref}

\newtheorem{theorem}{Theorem}

\theoremstyle{definition}
\newtheorem{definition}[theorem]{Definition}

\newtheorem{question}[theorem]{Question}

\theoremstyle{remark}

\begin{document}

\title[Answering an open question in fuzzy metric spaces]
{Answering an open question in fuzzy metric spaces}

\author[X. Wu]{Xinxing Wu}
\address[X. Wu]{School of Sciences, Southwest Petroleum University, Chengdu, Sichuan 610500, China}
\email{wuxinxing5201314@163.com}

\author[G. Chen]{Guanrong Chen}
\address[G. Chen]{Department of Electrical Engineering, City University of Hong Kong,
Hong Kong SAR, China}
\email{gchen@ee.cityu.edu.hk}

\thanks{This work was supported by the National Natural Science Foundation of China
(No. 11601449), the Science and Technology Innovation Team of Education Department of Sichuan for
Dynamical System and its Applications (No. 18TD0013), and the Youth Science and Technology Innovation
Team of Southwest Petroleum University for Nonlinear Systems (No. 2017CXTD02).}

\subjclass[2010]{03E72, 54H20.}

\date{\today}


\keywords{Fuzzy metric space; $\mathbb{R}$-uniform continuity.}

\begin{abstract}
This paper answers affirmatively Problem 32 posted in
\cite{GMM2012}, proving that, for every stationary fuzzy metric space $(X, M, *)$, the function $M_y(x):=M(x,y)$
defined therein is $\mathbb{R}$-uniformly continuous for all $y\in X$, and furthermore proves that the function $M$ is
$\mathbb{R}$-uniformly continuous.
\end{abstract}

\maketitle

During the last three decades, many research works were devoted to the notion of fuzzy metric spaces in different ways
\cite{D1982,E1979,K1984}. To obtain a Hausdorff topology of a fuzzy metric space, George and Veeramani~\cite{GV1994}
modified the concept of fuzzy metric space initiated by Kramosil and Michalek \cite{KM1975} and introduced the fuzzy
metric space in a new sense as follows.

\begin{definition}\label{Def-1}{\rm  (George and Veeramani \cite{GV1994}). A {\it fuzzy metric space} is an ordered triple
$(X, M, \ast)$ such that $X$ is a nonempty set, $\ast$ is a continuous $t$-norm, and $M$ is a fuzzy set
on $X\times X\times (0, +\infty)$ satisfying the following conditions, for all $x, y, z\in X$, $s, t>0$:
\begin{itemize}
\item[(GV1)] $M(x, y, t)>0$;

\item[(GV2)] $M(x, y, t)=1$ if and only if $x=y$;

\item[(GV3)] $M(x, y, t)=M(y, x, t)$;

\item[(GV4)] $M(x, y, t)\ast M(y, z, s)\leq M(x, z, t + s)$;

\item[(GV5)] $M(x, y, \_ ): (0, +\infty) \longrightarrow (0, 1]$ is continuous.
\end{itemize}
$(M, \ast)$, or simply $M$, is usually called a {\it fuzzy metric} on $X$,
and each value $M(x, y, t)$ can be understood as the nearness degree between $x$
and $y$ with respect to $t$. }
\end{definition}

According to Grabiec~\cite{G1989}, the real function $M(x, y, \_ )$ of Axiom (GV5) is nondecreasing for all $x, y\in X$.
George and Veeramani~\cite{GV1994} proved that every fuzzy metric space $(X, M, \ast)$ generates a topology $\mathscr{T}_{M}$
on $X$, which has as a base the family of open sets of the form $\{B_{M}(x, \varepsilon, t): x\in X, 0<\varepsilon<1, t>0\}$, where
$B_{M}(x, \varepsilon, t)=\{y\in X: M(x, y, t)>1-\varepsilon\}$.

Let $(X, d)$ be a metric space and $M_{d}$ be a function on $X\times X\times (0, +\infty)$ defined by
$$
M_{d}(x, y, t)=\frac{t}{t+d(x, y)}, \text{ for } (x, y, t)\in X\times X\times (0,+\infty).
$$
It can be verified that $(X, M_{d}, \cdot)$ is a fuzzy metric space, and $M_{d}$ is called the {\it standard fuzzy metric}
induced by $d$. This shows that every ordinary metric induces a fuzzy metric in the sense of George and Veeramani.

\begin{definition}{\rm \cite{GR2004}
A fuzzy metric $M$ on $X$ is {\it stationary} if, for any $x, y\in X$,
the function $M(x, y, \_ )$ is constant. In this case, write $M(x, y)$ instead of $M(x, y, \_ )$.}
\end{definition}

\begin{definition}{\rm (Gregori et al. \cite{GRS2001}).
Let$(X, M, \ast)$ be a fuzzy metric space. A mapping $f: X\longrightarrow \mathbb{R}$ is {\it $\mathbb{R}$-uniformly
continuous} if, for any $\varepsilon>0$, there exist $s>0$ and $\delta\in (0, 1)$ such that $M(x, y, s)>1-\delta$
implies $|f(x)-f(y)|<\varepsilon$.}
\end{definition}

Recently, Gregori et al. \cite{GMM2012} studied the fuzzy metric defined by
$M(x, y, t)=\frac{\min\{x, y\}+t}{\max\{x, y\}+t}$ and other fuzzy metrics related to it. In particular, they \cite{GMM2012}
proposed the following question for the real function $M(x, y)$ on stationary fuzzy metric spaces.

\begin{question}\label{Q-1}{\rm \cite[Problem~32]{GMM2012}
Let $(X, M, *)$ be a stationary fuzzy metric space. Is the real function $M_y(x):=M(x,y)$,
for each $x\in X$, $\mathbb{R}$-uniformly continuous for all $y\in X$?}
\end{question}

The following theorem proves that the function $M_{y}$ is $\mathbb{R}$-uniformly continuous for stationary fuzzy metric spaces,
giving a positive answer to Question~\ref{Q-1}. Denote $\mathbb{N}=\left\{1, 2, 3, \ldots\right\}$.

\begin{theorem}\label{Thm-continuous}
Let $(X, M, \ast)$ be a stationary fuzzy metric space. Then, for any $y\in X$, the function $M_y(x)=M(x,y)$
is $\mathbb{R}$-uniformly continuous.
\end{theorem}
\begin{proof}
Fix any $y\in X$, suppose on the contrary that $M_y$ is not $\mathbb{R}$-uniformly continuous. Then, there exists
$\varepsilon_0>0$ such that, for any $n\in \mathbb{N}$, there exist $x_n, z_n\in X$ with $M(x_n, z_n)>1-\frac{1}{n}$
satisfying $|M_{y}(x_n)-M_{y}(z_n)|\geq \varepsilon_0$. From $M_{y}(x_n), M_{y}(z_n)\in [0, 1]$, it follows that there exists
an increasing sequence $\{n_k\}_{k=1}^{+\infty}\subset \mathbb{N}$ such that
$\lim_{k\to +\infty} M_{y}(x_{n_k})=\xi$ and $\lim_{k\to +\infty} M_{y}(z_{n_k})=\eta$.
Clearly, $|\xi-\eta|\geq \varepsilon_0$.

\medskip

Consider the following two cases:
\begin{itemize}
\item[Case 1.] If $\xi>\eta$, without loss of generality, assume that $M_{y}(x_{n_k})>M_{y}(z_{n_k})$ for all $k\in \mathbb{N}$. This,
together with $M(y, x_{n_k})\ast M(x_{n_k}, z_{n_k})\leq M(y, z_{n_k})$, implies that
$$
M(y, x_{n_k})\ast M(x_{n_k}, z_{n_k})-M(y, x_{n_k})\leq M(y, z_{n_k})- M(y, x_{n_k})=
M_{y}(z_{n_k})-M_{y}(x_{n_k})\leq -\varepsilon_0.
$$
Then,
$$
0=\xi\ast 1-\xi=\lim_{k\to +\infty}(M(y, x_{n_k})\ast M(x_{n_k}, z_{n_k})-M(y, x_{n_k}))\leq  -\varepsilon_0,
$$
which is a contradiction.
\item[Case 2.] If $\xi<\eta$, without loss of generality, assume that $M_{y}(x_{n_k}) < M_{y}(z_{n_k})$ for all $k\in \mathbb{N}$. This,
together with $M(y, z_{n_k})\ast M(z_{n_k}, x_{n_k})\leq M(y, x_{n_k})$, implies that
$$
\varepsilon_0 \leq M_{y}(z_{n_k})-M_{y}(x_{n_k})= M(y, z_{n_k})-M(y, x_{n_k}) \leq
M(y, z_{n_k})-M(y, z_{n_k})\ast M(z_{n_k}, x_{n_k}).
$$
Then,
$$
\varepsilon_0 \leq \lim_{k\to +\infty}(M(y, z_{n_k})-M(y, z_{n_k})\ast M(z_{n_k}, x_{n_k}))=\eta-\eta\ast 1=0,
$$
which is a contradiction.
\end{itemize}
Therefore, $M_{y}$ is $\mathbb{R}$-uniformly continuous.
\end{proof}

Next, let $(X, M, \ast)$ be a stationary fuzzy metric space and $\mathscr{T}_{M}$ be the topology generated by
the fuzzy metric $M$. Define $\hat{M}: X^2 \times X^2 \times (0, +\infty)\longrightarrow [0, 1]$ by
$$
\hat{M} ((x_1, x_2), (y_1, y_2), t)=\min\left\{M(x_1, y_1), M(x_2, y_2)\right\},
\text{ for } ((x_1, x_2), (y_1, y_2), t)\in X^2\times X^2\times (0, +\infty).
$$
From \cite[Proposition~4.1, Theorem~4.4]{MS2016}, it follows that
\begin{enumerate}[(1)]
  \item $(\hat{M}, \ast)$ is a fuzzy metric on $X^2$, i.e., $(X^2, \hat{M}, \ast)$
is a fuzzy metric space;
  \item the topology $\mathscr{T}_{\hat{M}}$ generated by $\hat{M}$ is the product topology for $X^2$.
\end{enumerate}

Similarly to Theorem~\ref{Thm-continuous}, the following theorem proves the $\mathbb{R}$-uniform
continuity of $M$ on $(X^2, \hat{M}, \ast)$.

\begin{theorem}
Let $(X, M, \ast)$ be a stationary fuzzy metric space. Then, the function
\begin{eqnarray*}
M: (X^2, \hat{M}, \ast) & \longrightarrow & (0, 1],\\
(x, y) & \longmapsto & M(x, y),
\end{eqnarray*}
is $\mathbb{R}$-uniformly continuous.
\end{theorem}
\begin{proof}
Suppose on the contrary that $M$ is not $\mathbb{R}$-uniformly continuous, which means that there exists
$\varepsilon_0>0$ such that, for any $n\in \mathbb{N}$, there exist $(x_n, z_n), (x'_n, z'_n)\in X^2$ with
$\hat{M}((x_n, z_n), (x'_n, z'_n))>1-\frac{1}{n}$ satisfying $|M(x_n, z_n)-M(x'_n, z'_n)|\geq \varepsilon_0$.
From $M(x_n, z_n), M(x'_n, z'_n)\in [0, 1]$, it follows that there exists an increasing sequence
$\{n_k\}_{k=1}^{+\infty}\subset \mathbb{N}$ such that
$\lim_{k\to +\infty} M(x_{n_k}, z_{n_k})=\xi$ and $\lim_{k\to +\infty} M(x'_{n_k}, z'_{n_k})=\eta$.
Clearly, $|\xi-\eta|\geq \varepsilon_0$.

\medskip

Consider the following two cases:
\begin{itemize}
\item[Case 1.] If $\xi>\eta$, without loss of generality, assume that $M(x_{n_k}, z_{n_k})>M(x'_{n_k}, z'_{n_k})$
for all $k\in \mathbb{N}$. This, together with $M(x_{n_k}, z_{n_k})\ast M(z'_{n_k}, z_{n_k}) \ast M(x'_{n_k}, x_{n_k})
\leq M(x'_{n_k}, z'_{n_k})$, implies that
$$
M(x_{n_k}, z_{n_k})\ast M(z'_{n_k}, z_{n_k}) \ast M(x'_{n_k}, x_{n_k})-M(x_{n_k}, z_{n_k})\leq M(x'_{n_k}, z'_{n_k})-M(x_{n_k}, z_{n_k})
\leq -\varepsilon_0.
$$
Then,
$$
0=\xi\ast 1 \ast 1-\xi=\lim_{k\to +\infty}(M(x_{n_k}, z_{n_k})\ast M(z'_{n_k}, z_{n_k}) \ast M(x'_{n_k}, x_{n_k})-M(x_{n_k}, z_{n_k}))
\leq  -\varepsilon_0,
$$
which is a contradiction.
\item[Case 2.] If $\xi<\eta$, without loss of generality, assume that $M(x_{n_k}, z_{n_k}) < M(x'_{n_k}, z'_{n_k})$ for all $k\in \mathbb{N}$. This,
together with $M(x'_{n_k}, z'_{n_k})\ast M(z'_{n_k}, z_{n_k}) \ast M(x'_{n_k}, x_{n_k})\leq M(x_{n_k}, z_{n_k})$, implies that
$$
\varepsilon_0 \leq M(x'_{n_k}, z'_{n_k})-M(x_{n_k}, z_{n_k}) \leq M(x'_{n_k}, z'_{n_k})-M(x'_{n_k}, z'_{n_k})\ast M(z'_{n_k}, z_{n_k})
\ast M(x'_{n_k}, x_{n_k}).
$$
Then,
$$
\varepsilon_0 \leq \lim_{k\to +\infty}( M(x'_{n_k}, z'_{n_k})-M(x'_{n_k}, z'_{n_k})\ast M(z'_{n_k}, z_{n_k}) \ast M(x'_{n_k}, x_{n_k}))
=\eta -\eta \ast 1\ast 1=0,
$$
which is a contradiction.
\end{itemize}
Therefore, $M$ is $\mathbb{R}$-uniformly continuous.
\end{proof}


\bibliographystyle{amsplain}

\end{document}